\documentclass[12pt]{article}
\usepackage{amsthm}
\usepackage{tikz}
\usepackage{fullpage}

\newtheorem{thm}{Theorem}
\newtheorem{lem}[thm]{Lemma}
\newtheorem{cor}[thm]{Corollary}

\theoremstyle{definition}
\newtheorem{rem}[thm]{Remark}

\begin{document}
\title{A (forgotten) upper bound for the spectral radius of a graph}
\author{Clive Elphick\thanks{\texttt{clive.elphick@gmail.com}},~~~~~~
Chia-an Liu\thanks{\texttt{twister.imm96g@g2.nctu.edu.tw}}}
\maketitle

\abstract{The best degree-based upper bound for the spectral radius is due to Liu and Weng. This paper begins by demonstrating that a (forgotten) upper bound for the spectral radius dating from 1983 is equivalent to their much more recent bound. This bound is then used to compare lower bounds for the clique number.
A series of line graph based upper bounds for the Q-index is then proposed and compared experimentally with a graph based bound.
Finally a new lower bound for generalised $r-$partite graphs is proved, by extending a result due to Erd\"os.}

\section{Introduction}

Let $G$ be a simple and undirected graph with $n$ vertices, $m$ edges, and degrees $\Delta = d_1 \ge d_2 \ge ... \ge d_n = \delta$. Let $d$ denote the average vertex degree, $\omega$ the clique number and $\chi$ the chromatic number. Finally let $\mu(G)$ denote the spectral radius of $G$, $q(G)$ denote the spectral radius of the signless Laplacian of $G$ and $G^L$ denote the line graph of $G$.

In 1983, Edwards and Elphick \cite{edwards83} proved in their Theorem 8 (and its corollary) that $\mu \le y - 1$, where $y$ is defined by the equality:

\begin{equation}\label{eq:edwards}
y(y - 1) = \sum_{k = 1}^{\lfloor{y}\rfloor} d_k + (y - \lfloor{y}\rfloor)d_{\lceil{y}\rceil}.
\end{equation}

Edwards and Elphick \cite{edwards83} show that $1 \le y \le n$ and that $y$ is a single-valued function of $G$.

This bound is exact for regular graphs because, we then have that:

\[
d = \mu \le y - 1 = \frac{1}{y}\left(\sum_{k = 1}^{\lfloor{y}\rfloor} d + (y - \lfloor{y}\rfloor)d\right) = d.
\]

The bound is also exact for various bidegreed graphs. For example, let G be the Star graph on $n$ vertices, which has $\mu =\sqrt{n - 1}$. It is easy to show that $\lfloor \sqrt{n - 1} \rfloor < y < \lceil \sqrt{n - 1} \rceil$.  It then follows that $y$ is the solution to the equation:

\[
y(y - 1) = (n - 1) + \lfloor \sqrt{n - 1} \rfloor - 1 + (y - \lfloor \sqrt{n - 1} \rfloor) = n - 2 + y,
\]
which has the solution $y = 1 + \sqrt{n - 1}$, so $\mu \le y - 1 = \sqrt{n - 1}$.

Similarly let G be the Wheel graph on $n$  vertices, which has $\mu = 1 + \sqrt{n}$. It is straightforward to show that $y = 2 + \sqrt{n}$ is the solution to (\ref{eq:edwards}) so again the bound is exact.

\section{An upper bound for the spectral radius}

The calculation of $y$ can involve a two step process.

1. Restrict $y$ to integers, so (\ref{eq:edwards}) simplifies to:

\[
y(y - 1) = \sum_{k=1}^y d_k.
\]

Since $d \le \mu$, we can begin with $y = \lfloor d + 1 \rfloor$, and then increase $y$ by unity until $y(y - 1) \ge \sum_{k=1}^y d_k$.  This determines that either $y = a$ or $a < y < a + 1$, where $a$ is an integer.

2. Then, if necessary,  solve the following quadratic equation:

\begin{equation}\label{eq:elphick}
y(y - 1) = \sum_{k=1}^a d_k + (y - a)d_{a+1}.
\end{equation}

For convenience let $c = \sum_{k=1}^a d_k$. Equation (\ref{eq:elphick}) then becomes:

\[
y^2 - y(1 + d_{a+1}) - (c - ad_{a+1}) = 0.
\]

Therefore

\[
y = \frac{d_{a+1} + 1 + \sqrt{(d_{a+1} + 1)^2 + 4(c - ad_{a+1})}}{2}
\]

so

\[
\mu \le y - 1 = \frac{d_{a+1} - 1 + \sqrt{(d_{a+1} + 1)^2 + 4(c - ad_{a+1})}}{2}.
\]

This two step process can be combined as follows, by letting $a + 1 = k$:

\begin{equation}\label{eq:elphick2}
\mu \le  \frac{d_k - 1 + \sqrt{(d_k + 1)^2 + 4\sum_{i=1}^{k-1}(d_i - d_k)}}{2}, \mbox{  where  } 1 \le k \le n.
\end{equation}

In 2012, Liu and Weng \cite{liu2013} proved (\ref{eq:elphick2}) using a different approach. They also proved there is equality if and only if $G$ is regular or there exists $2 \le t \le k$ such that $d_1 = d_{t-1} = n - 1$ and $d_t = d_n$. Note that if $k = 1$ this reduces to $\mu \le \Delta$.

If we set $k = n$ in (\ref{eq:elphick2}) then:

\[
 \mu \le \frac{\delta - 1 + \sqrt{(\delta + 1)^2 - 4n\delta + 8m}}{2}
\]

which was proved by Nikiforov \cite{nikiforov02} in 2002.

\section{Lower bounds for the clique number}

Tur\'an's Theorem, proved in 1941, is a seminal result in extremal graph theory. In its concise form it states that:

\[
\frac{n}{n - d} \le \omega(G).
\]

Edwards and Elphick \cite{edwards83} used $y$ to prove the following lower bound for the clique number:

\begin{equation}\label{eq:chris}
\frac{n}{n - y + 1} < \omega(G) + \frac{1}{3}.
\end{equation}

In 1986, Wilf \cite{wilf86} proved that:

\[
\frac{n}{n - \mu} \le \omega(G).
\]

Note, however, that:

\[
\frac{n}{n - y + 1} \not\le \omega(G),
\]

since for example $\frac{n}{n - y + 1} = 2.13$ for  $K_{7,9}$ and $\frac{n}{n - y + 1} = 3.1$ for $K_{3,3,4}$.

Nikiforov \cite{nikiforov02} proved a conjecture due to Edwards and Elphick \cite{edwards83} that:

\begin{equation}\label{eq:niki}
\frac{2m}{2m - \mu^2} \le \omega(G).
\end{equation}

Experimentally, bound (\ref{eq:niki}) performs better than bound (\ref{eq:chris}) for most graphs.

\section{Upper bounds for the Q-index}
Let $q(G)$ denote the spectral radius of the signless Laplacian of $G$. In this section we investigate graph and line graph based bounds for $q(G)$ and then compare them experimentally.

\subsection{Graph bound}

Nikiforov \cite{nikiforov14} has recently strengthened various upper bounds for $q(G)$ with the following theorem.

\begin{thm}
If $G$ is a graph with $n$ vertices, $m$ edges, with maximum degree $\Delta$ and minimum degree $\delta$, then

\[
q(G) \le min\left(2\Delta, \frac{1}{2}\left(\Delta + 2\delta - 1 + \sqrt{(\Delta + 2\delta - 1)^2 + 16m -8(n - 1 + \Delta)\delta}\right)\right).
\]

Equality holds if and only if $G$ is regular or $G$ has a component of order $\Delta + 1$ in which every vertex is of degree $\delta$ or $\Delta$, and all other components are $\delta$-regular.

\end{thm}

\subsection{Line graph bounds}

The following well-known Lemma (see, for example, Lemma 2.1 in \cite{chen02}) provides an equality between the spectral radii of
the signless Laplacian matrix and the adjacency matrix of the line graph of a graph.

\begin{lem}\label{lem:chen02}
If $G^L$ denotes the line graph of $G$ then:
\begin{equation}\label{eq:chen02}
q(G) = 2 + \mu(G^L).
\end{equation}
\end{lem}

\medskip

Let $\Delta_{ij}=\{d_{i}+d_{j}-2 ~|~ i \sim j \}$
be the degrees of vertices in $G^{L},$
and $\Delta_{1} \geq \Delta_{2} \geq \ldots \geq \Delta_{m}$ be a renumbering of them in non-increasing order.
Cvetkovi\'c \emph{et al.} proved the following theorem using Lemma~\ref{lem:chen02}.
\begin{thm}\label{thm:cvetkovic07}(Theorem 4.7 in \cite{cvetkovic07})
$$q(G) \leq 2+\Delta_{1}$$
with equality if and only if $G$ is regular or semi-regular bipartite.
\end{thm}

\medskip

The following lemma is proved in varying ways in \cite{shu04,das11,liu2013}.
\begin{lem}\label{lem:shu04}
$$\mu(G) \leq \frac{d_{2}-1+\sqrt{(d_{2}-1)^{2}+4d_{1}}}{2}$$
with equality if and only if $G$ is regular or $n-1=d_{1}>d_{2}=d_{n}.$
\end{lem}

\medskip

Chen \emph{et al.} combined Lemma~\ref{lem:chen02} and Lemma~\ref{lem:shu04} to prove the following result.
\begin{thm}\label{thm:chen11}(Theorem 3.4 in \cite{chen11})
$$q(G) \leq 2+\frac{\Delta_{2}-1+\sqrt{(\Delta_{2}-1)^{2}+4\Delta_{1}}}{2}$$
with equality if and only if $G$ is regular, or semi-regular bipartite, or the tree obtained by joining
an edge to the centers of two stars $K_{1,\frac{n}{2}-1}$ with even $n,$ or $n-1=d_{1}=d_{2} > d_{3}=d_{n}=2.$
\end{thm}

\medskip

Stating (\ref{eq:elphick2}) as a Lemma we have:
\begin{lem}\label{lem:edwards83}
For $1 \leq k \leq n,$
\begin{equation}\label{eq:dwards83}
\mu(G) \leq \phi_{k} := \frac{d_{k}-1+\sqrt{(d_{k}+1)^{2}+4\sum_{i=1}^{k-1}(d_{i}-d_{k})}}{2}
\end{equation}
with equality if and only if $G$ is regular or there exists $2 \leq t \leq k$ such that
$n-1=d_{1}=d_{t-1}>d_{t}=d_{n}.$ Furthermore,
$$\phi_{\ell}=\min\{\phi_{k} ~|~ 1 \leq k \leq n \}$$
where $3 \leq \ell \leq n$ is the smallest integer such that $\sum_{i=1}^{\ell}d_{i}<\ell(\ell-1).$
\end{lem}

\medskip

Combining Lemma~\ref{lem:chen02} and Lemma~\ref{lem:edwards83}
provides the following series of upper bounds for the signless Laplacian spectral radius.
\begin{thm}\label{thm:signlessLap}
For $1 \leq k \leq m,$ we have
\begin{equation}\label{eq:signlessLap}
q(G) \leq \psi_{k} := 1 + \frac{\Delta_{k}+1+
\sqrt{(\Delta_{k}+1)^{2}+4\sum_{i=1}^{k-1}(\Delta_{i}-\Delta_{k})}}{2}
\end{equation}
with equality if and only if
$\Delta_{1}=\Delta_{m}$ or there exists $2 \leq t \leq k$ such that
$m-1=\Delta_{1}=\Delta_{t-1} > \Delta_{t}=\Delta_{m}.$ Furthermore,
$$\psi_{\ell}=\min\{\psi_{k} ~|~ 1 \leq k \leq m \}$$
where $3 \leq \ell \leq m$ is the smallest integer such that $\sum_{i=1}^{\ell}\Delta_{i}<\ell(\ell-1).$
\end{thm}
\begin{proof}
$G^{L}$ is simple.
Hence (\ref{eq:signlessLap}) is a direct result of (\ref{eq:chen02}) and (\ref{eq:dwards83}).
The sufficient and necessary conditions are immediately those in Lemma~\ref{lem:edwards83}.
\end{proof}

\medskip

\begin{rem}
Note that Theorem~\ref{thm:signlessLap} generalizes both Theorem~\ref{thm:cvetkovic07} and Theorem~\ref{thm:chen11}
since these bounds are precisely $\psi_{1}$ and $\psi_{2}$ in (\ref{eq:signlessLap}) respectively.

We list all the extremal graphs with equalities in (\ref{eq:signlessLap}) in the following.
From Theorem~\ref{thm:cvetkovic07} the graphs with $q(G)=\psi_{1},$
i.e. $\Delta_{1}=\Delta_{m},$ are regular or semi-regular bipartite.

From Theorem~\ref{thm:chen11} the graphs
with $q(G) < \psi_{1}$ and $q(G)=\psi_{2},$ i.e. $m-1=\Delta_{1}>\Delta_{2}=\Delta_{m},$
are the tree obtained by joining an edge to the centers of two stars $K_{1,\frac{n}{2}-1}$ with even n,
or $n-1=d_{1}=d_{2}>d_{3}=d_{n}=2.$

The only graph with $q(G) < \min\{\psi_{i}|i=1,2\}$ and $q(G)=\psi_{3},$ i.e.
$m-1=\Delta_{1}=\Delta_{2}>\Delta_{3}=\Delta_{m},$ is the $4$-vertex graph $K_{1,3}^{+}$
obtained by adding one edge to $K_{1,3}.$
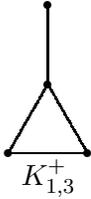
\begin{picture}(50,75)
\put(10,14){\circle*{3}} \put(40,14){\circle*{3}}
\put(25,40){\circle*{3}} \put(25,70){\circle*{3}}
\qbezier(10,14)(25,14)(40,14)
\qbezier(10,14)(17.5,27)(25,40)
\qbezier(40,14)(32.5,27)(25,40)
\qbezier(25,40)(25,55)(25,70)
\put(15,2){$K_{1,3}^{+}$}
\end{picture}

We now prove that no graph satisfies $q(G) < \min\{\psi_{i}|1\leq i \leq k - 1\}$ and $q(G)=\psi_{k}$ where $m \geq k \geq 4.$ Let $G$ be a counter-example such that $m-1=\Delta_{1}=\Delta_{k-1}>\Delta_{k}=\Delta_{m}.$
Since $\Delta_{3}=m-1$ there are at least $3$ edges incident to all other edges in $G.$
If these $3$ edges form a $3$-cycle then there is nowhere to place the fourth edge,
which is a contradiction. Hence they are incident to a common vertex, and $G$ has to be a star graph.
However a star graph is semi-regular bipartite so $q(G)=\psi_{1},$ which completes the proof.
\end{rem}

\medskip

\begin{rem}

By analogy with (\ref{eq:edwards}), if $z$ is defined by the equality
\[
z(z - 1) = \sum_{k = 1}^{\lfloor{z}\rfloor} \Delta_k + (z - \lfloor{z}\rfloor)\Delta_{\lceil{z}\rceil},
\]
then $q \le z + 1$.
This bound is exact for $d-$regular graphs, because we then have:

\[
2d = q \le z + 1 = 2 + (z - 1) = 2 + \frac{1}{z}\left(\sum_{k = 1}^{\lfloor{z}\rfloor} \Delta + (z - \lfloor{z}\rfloor)\Delta\right) = 2 +\Delta = 2d.
\]
\end{rem}
\subsection{Experimental comparison}
It is straightforward to compare the above bounds experimentally using the named graphs and LineGraph function in Wolfram Mathematica. Theorem 1 is exact for some graphs (eg Wheels) for which Theorems 5 and 7 are inexact  and Theorems 5 and 7 are exact for some graphs (eg complete bipartite) for which Theorem 1 is inexact. Tabulated below are the numbers of named irregular graphs on 10, 16, 25 and 28 vertices in Mathematica and the average values of $q$ and the bounds in Theorems 1, 5 and 7.

\[
\begin{array}{cccccc}
\mbox{n} 		& \mbox{irrregular graphs}	 & \mbox{q(G)} 	& \mbox{Theorem 1} 	& \mbox{Theorem 5} 	& \mbox{Theorem 7} 	\\
10			&59					&9.3			&10.0			&10.3			&9.8				\\
16			&48					&10.3		&11.2			&11.5			&11.0			\\
25			&25					&11.5		&13.4			&13.1			&12.6			\\
28			&21					&11.2		&12.6			&12.7			&12.2			\\	
\end{array}
\]

It can be seen that Theorem 5 gives results that are broadly equal on average to Theorem 1 and Theorem 7 gives results which are on average modestly better. This is unsurprising since more data is involved in Theorem 7 than in the other two theorems. For some graphs, $q(G)$ is minimised in Theorem  7 with large values of $k$.

\section{A lower bound for the Q-index}

Elphick and Wocjan \cite{elphick14} defined a measure of graph irregularity, $\nu$, as follows:

\[
\nu = \frac{n\sum d_i^2}{4m^2},
\]

where $\nu \ge 1$, with equality only for regular graphs.

It is well known that $q \ge 2\mu$ and Hofmeister \cite{hofmeister88} has proved that $\mu^2 \ge \sum d_i^2/n$, so it is immediate that:

\[
q \ge 2\mu \ge \frac{4m\sqrt{\nu}}{n}.
\]

Liu and Liu \cite{liu09} improved this bound in the following theorem, for which we provide a simpler proof using Lemma 2.

\begin{thm}
Let $G$ be a graph with irregularity $\nu$ and Q-index $q(G)$. Then

\[
q(G) \ge \frac{4m\nu}{n}.
\]

This is exact for complete bipartite graphs.

\end{thm}

\begin{proof}

Let $G^L$ denote the line graph of $G$. From Lemma 2 we know that $q(G) = 2 + \mu(G^L)$ and it is well known that $n(G^L) = m$ and $m(G^L) = (\sum d_i^2/2) - m$. Therefore:

\[
q = 2 + \mu(G^L) \ge 2 + \frac{2m(G^L)}{n(G^L)} = 2 + \frac{2}{m}\left(\frac{\sum d_i^2}{2} - m\right) = \frac{\sum d_i^2}{m} = \frac{4m\nu}{n}.
\]

For the complete bipartite graph $K_{s,t}$ :
\[
q \ge \frac{\sum_i d_i^2}{m} = \frac{\sum_{ij\in E} (d_i + d_j)}{m} = d_i + d_j = s + t = n, \mbox{ which is exact}.
\]

\end{proof}

\section{Generalised $r-$partite graphs}

In a series of papers, Bojilov and others have generalised the concept of an $r-$partite graph. They define the parameter $\phi$ to be the smallest integer $r$ for which $V(G)$ has an $r-$partition:

\[
V(G) = V_1\cup V_2 \cup ... \cup V_r, \mbox{  such that  } d(v) \le n - n_i, \mbox{  where  } n_i = |V_i|,
\]
for all $v \in V_i$ and for $i = 1,2, ... ,r$.

Bojilov \emph{et al} \cite{bojilov13} proved that $\phi(G) \le \omega(G)$ and Khadzhiivanov and Nenov \cite{khad04} proved that:

\[
\frac{n}{n - d} \le \phi(G).
\]

Despite this bound, Elphick and Wocjan \cite{elphick14} demonstrated that:

\[
\frac{n}{n - \mu} \not\le \phi(G).
\]

However, it is proved below in Corollary 10 that:

\[
\frac{n}{n - \mu} \le \frac{n}{n - y +1} < \phi(G)+ \frac{1}{3}.
\]

\emph{Definition}

If $H$ is any graph of order $n$ with degree sequence $d_H(1) \ge d_H(2) \ge ... \ge d_H(n)$, and if $H^*$ is any graph of order $n$ with degree sequence $d_{H^*}(1) \ge d_{H^*}(2) \ge ... \ge d_{H^*}(n)$, such that $d_H(i) \le d_{H^*}(i)$ for all $i$, then $H^*$ is said to "dominate" $H$.

Erd\"os proved that if $G$ is any graph of order $n$, then there exists a graph $G^*$ of order $n$, where $\chi(G^*) = \omega(G) = r$, such that $G^*$ dominates $G$ and $G^*$ is complete $r-$partite.

\begin{thm}

If $G$ is any graph of order $n$, then there exists a graph $G^*$ of order $n$, where $\omega(G^*) = \phi(G) = r$, such that $G^*$ dominates $G$, and $G^*$ is complete $r-$partite.

\end{thm}

\begin{proof}

Let $G$ be a generalised $r-$partite graph with $\phi(G) = r$ and $n_i = |V_i|$, and let $G^*$ be the complete $r-$partite graph $K_{n_1, ... , n_r}$. Let $d(v)$ denote the degree of vertex $v$ in $G$ and $d^*(v)$ denote the degree of vertex $v$ in $G^*$. Clearly $\chi(G^*) = \omega(G^*) = r$, and by the definition of a generalised $r-$partite graph:

\[
d^*(v) = n - n_i \ge d(v)
\]
for all $v \in V_i$ and for $i = 1, ..., r$. Therefore $G^*$ dominates $G$.

\end{proof}

\begin{lem} (Lemma 4 in \cite{edwards83})

Assume $G^*$ dominates $G$. Then $y(G^*) \ge y(G)$.

\end{lem}

\begin{thm}

\[
\frac{n}{n - y(G) +1} < \phi(G) + \frac{1}{3}.
\]

\end{thm}

\begin{proof}

Let $G^*$ be any graph of order $n$, where $\omega(G^*) = \phi(G)$ such that $G^*$ dominates $G$. (By Theorem 7 at least one such graph $G^*$ exists.) Then, using Lemma 8:

\[
\frac{n}{n - y(G) + 1} \le \frac{n}{n - y(G^*) + 1} < \omega(G^*) + \frac{1}{3} = \phi(G) + \frac{1}{3} \le \omega(G) + \frac{1}{3}.
\]

\end{proof}

\begin{cor}

\[
\frac{n}{n - \mu} < \phi(G) + \frac{1}{3}.
\]

\end{cor}

\begin{proof}
Immediate since $\mu \le y - 1$.

\end{proof}

\end{document}